\documentclass{amsart}
\usepackage[T1]{fontenc}
\usepackage[utf8]{inputenc}
\usepackage{lmodern}
\usepackage{amsmath,amssymb,amsthm,mathtools,mathabx}
\usepackage{bm}
\usepackage[margin=1.2in]{geometry}
\usepackage{microtype}
\usepackage{tikz}
\usetikzlibrary{cd}
\usepackage{multicol}
\usepackage{array}
\usepackage[shortlabels]{enumitem}
\usepackage[hyphens]{url} 
\usepackage[unicode,colorlinks,
 	citecolor=blue,
 	urlcolor=blue,
 	linkcolor=blue
        ]{hyperref} 

\usepackage{nicefrac} 

\newcommand\ip[2]{\langle #1, #2 \rangle}
\newcommand\seq[1]{\langle #1 \rangle}

\newcommand{\defeq}{\vcentcolon=}

\DeclareMathOperator{\dom}{dom}
\DeclareMathOperator{\range}{range}

\DeclareMathOperator{\uh}{\upharpoonright}

\def\ZFC{\mathsf{ZFC}}
\def\ZF{\mathsf{ZF}}

\def\AD{\mathsf{AD}}
\def\Dec{\mathsf{Dec}}

\let\0\varnothing

\let\phi\varphi
\let\term\emph

\renewcommand{\epsilon}{\varepsilon}

\newcommand{\concat}{^\smallfrown}

\theoremstyle{plain}
\newtheorem{theorem}{Theorem}

\newtheorem*{claim*}{Claim}
\newtheorem{observation}[theorem]{Observation}

\newtheorem*{example*}{Example}

\newtheorem*{proposition*}{Proposition}

\newtheorem*{fact*}{Fact}

\newtheorem*{conjecture*}{Conjecture}

\newtheorem{lemma}[theorem]{Lemma}
\newtheorem*{lemma*}{Lemma}

\newtheorem*{remark*}{Remark}

\newtheorem{question}[theorem]{Question}
\newtheorem*{question*}{Question}
\theoremstyle{definition}\newtheorem{definition}[theorem]{Definition}
\theoremstyle{definition}
\theoremstyle{definition}

\numberwithin{theorem}{section}

\begin{document}
\title[The Solecki Dichotomy and the Posner-Robinson Theorem]{The Solecki Dichotomy and the Posner-Robinson Theorem are Almost Equivalent}

\author{Patrick Lutz}
\address{University of California, Los Angeles}
\email{pglutz@ucla.edu}

\begin{abstract}
The Solecki dichotomy in descriptive set theory and the Posner-Robinson theorem in computability theory bear a superficial resemblance to each other and can sometimes be used to prove the same results, but do not have any obvious direct relationship. We show that in fact there is such a relationship by formulating slightly weakened versions of the two theorems and showing that, when combined with determinacy principles, each one yields a short proof of the other. This relationship also holds for generalizations of the Solecki dichotomy and the Posner-Robinson theorem to higher levels of the Borel/hyperarithmetic hierarchy.
\end{abstract}

\maketitle

\section{Introduction}

This paper is about the relationship between two theorems: the Solecki dichotomy from descriptive set theory, which says that every Borel function on the reals is either a countable union of continuous functions or at least as complicated as the Turing jump \cite{pawlikowski2012decomposing}, and the Posner-Robinson theorem in computability theory, which says that every real is either computable or looks like $0'$ relative to some oracle \cite{posner1981degrees}. We will give formal statements of both theorems later.

Superficially, these theorems are very similar. Recall that every continuous function on the reals is computable relative to some oracle. So, allowing for some poetic license, we might summarize both theorems as saying that every object of some sort is either computable or at least as complicated as the jump.

However, it is not apparent whether this similarity is more than superficial. In the Solecki dichotomy, the objects under consideration are second-order---functions from the real numbers to the real numbers---while in the Posner-Robinson theorem they are first order---individual real numbers. Note that this difference is distinct from the observation that the Solecki dichotomy is a ``bold-face'' statement while the Posner-Robinson theorem is a ``light-face'' one. Additionally, the superficial analogy seems to suggest that the Solecki dichotomy should simply say that every function is either continuous (rather than a countable union of continuous functions) or at least as complicated as the jump, but this is false.

One indication that there might be something mathematically significant behind this similarity can be found in work by Kihara. First, Kihara observed \cite{lutz2021results} that the Solecki dichotomy could be used to prove a special case of Martin's conjecture that had previously been proved by Slaman and Steel in \cite{slaman1988definable} using the Posner-Robinson theorem. Second, work by Gregoriades, Kihara and Ng \cite{gregoriades2021turing} used a version of the Posner-Robinson theorem to prove results related to the decomposability conjecture that had previously been proved using the Solecki dichotomy \cite{pawlikowski2012decomposing, mottoros2013structure}.

The goal of this paper is to show that this is no accident---there is a meaningful technical relationship between the Solecki dichotomy and the Posner-Robinson theorem. In particular, we will formulate slightly weakened versions of the Solecki dichotomy and the Posner Robinson theorem\footnote{Both of which suffice to prove the results of Kihara and others mentioned above.} and show that each one can be used to give a short proof of the other\footnote{Although for one direction---proving the Posner-Robinson theorem using the Solecki dichotomy---we will need to use $\mathbf{\Pi}^1_1$-determinacy, which is not provable in $\ZFC$.}. The fact that this is possible, along with the details of the proofs, support the view that the Solecki dichotomy is morally (though not literally) a bold-face version of the Posner-Robinson theorem.

There are also generalizations of the Solecki dichotomy and the Posner-Robinson theorem to higher levels of the Borel/hyperarithmetic hierarchy and all of our results go through for these generalizations, with the proofs more or less unchanged. We discuss this further in Section \ref{sec:generalizations}.

In the remainder of the introduction, we will introduce the Solecki dichotomy and the Posner-Robinson theorem, as well as the weakened versions that we will use in this paper. We will also briefly discuss determinacy principles, which provide the main technical tool that we will use in our proofs.

\subsection*{The Solecki dichotomy}

Informally, the Solecki dichotomy says that every sufficiently definable function from reals to reals is either a countable union of continuous functions or at least as complicated as the Turing jump\footnote{Actually, most published statements of the Solecki dichotomy use a function called ``Pawlikowski's function'' in place of the Turing jump, but it is not hard to see that these two versions of the theorem are equivalent}. To state it formally, we must first state precisely what we mean by ``a countable union of continuous functions'' and ``at least as complicated as the Turing jump.''

\begin{definition}
A function $f \colon \omega^\omega \to \omega^\omega$ is \term{$\sigma$-continuous} if there is a partition $\{A_n\}_{n \in \omega}$ of $\omega^\omega$ into countably many pieces such that for each $n$, $f\uh_{A_n}$ is continuous with respect to the subspace topology on $A_n$.
\end{definition}

Note that there is a small subtlety here: just because $f\uh_{A_n}$ is continuous with respect to the subspace topology on $A_n$ does not mean that $f\uh_{A_n}$ can be extended to a continuous function defined on all of $\omega^\omega$. We will also refer to a partial function which is continuous with respect to the subspace topology on its domain as a \term{partial continuous function}.

\begin{definition}
A function $f\colon \omega^\omega \to \omega^\omega$ \term{topologically embeds} into a function $g\colon \omega^\omega \to \omega^\omega$ if there are topological embeddings $\phi, \psi \colon \omega^\omega \to \omega^\omega$ such that $\psi\circ f = g \circ \phi$. In other words, the following diagram commutes.
\begin{center}
\begin{tikzcd}
\omega^\omega \arrow[r, "g"]
& \omega^\omega\\
\omega^\omega \arrow[r, "f"]\arrow[u, "\phi"]
& \omega^\omega\arrow[u, "\psi"]
\end{tikzcd}
\end{center}
\end{definition}

\begin{definition}
The \term{Turing jump} is the function $J \colon \omega^\omega \to \omega^\omega$ defined by
\[
  J(x)(n) \defeq
  \begin{cases}
    0 &\text{if } \Phi^x_n(n)\uparrow\\
    m + 1 &\text{if } \Phi^x_n(n)\downarrow \text{ in exactly $m$ steps}.
  \end{cases}
\]
\end{definition}

Note that our official definition of the Turing jump is slightly different from the standard one, in which $J(x)(n)$ only indicates whether or not $\Phi^x_n(n)$ converges, not how many steps it takes to converge. This is necessary for Theorem~\ref{thm:solecki1}, but it doesn't matter anywhere else in this paper---in other words, after the statement of Theorem~\ref{thm:solecki1}, the entire remainder of the paper can be read as if we had defined $J$ in the usual way instead of the definition given above.

\begin{theorem}[Solecki dichotomy]
\label{thm:solecki1}
For every Borel function $f \colon \omega^\omega \to \omega^\omega$, either $f$ is $\sigma$-continuous or the Turing jump topologically embeds into $f$.
\end{theorem}

Theorem~\ref{thm:solecki1} was first proved by Solecki in \cite{solecki1998decomposing} in the special case where $f$ is of Baire class $1$. It was extended to all Borel functions by Zapletal in \cite{zapletal2004descriptive} and to all analytic functions by Pawlikowski and Sabok in \cite{pawlikowski2012decomposing}. It is also known to hold for all functions under $\AD$ \cite{zapletal2004descriptive}.

We will now state two weaker versions of the Solecki dichotomy, obtained by replacing topological embeddability with weaker notions of reducibility between functions.

\begin{definition}
\label{def:reducible}
A function $f \colon \omega^\omega \to \omega^\omega$ is \term{reducible}\footnote{This notion of reducibility has also been called \term{strong continuous Weihrauch reducibility} \cite{brattka2011weihrauch} and \term{continuous reducibility} (by Carroy \cite{carroy2013quasi}).} to a function $g \colon \omega^\omega \to \omega^\omega$, written $f \leq g$, if there are partial continuous functions $\phi, \psi \colon \omega^\omega \to \omega^\omega$ such that for all $x \in \omega^\omega$, $f(x) = \psi(g(\phi(x)))$. In other words, the following diagram commutes
\begin{center}
\begin{tikzcd}
\omega^\omega \arrow[r, "g"]
& \omega^\omega \arrow[d, "\psi"]\\
\omega^\omega \arrow[r, "f"] \arrow[u, "\phi"]
& \omega^\omega
\end{tikzcd}
\end{center}
\end{definition}

Note that this definition implies that $\phi$ is actually total and that $\range(g\circ\phi) \subseteq \dom(\psi)$.

\begin{theorem}[Solecki dichotomy, version 2]
\label{thm:solecki2}
For every Borel function $f \colon \omega^\omega \to \omega^\omega$, either $f$ is $\sigma$-continuous or $J \leq f$.
\end{theorem}

Note that if $f$ topologically embeds into $g$ then $f$ is also reducible to $g$: if $\phi$ and $\psi$ are topological embeddings such that $\psi \circ f = g \circ \phi$ then by definition, $\psi^{-1}$ is a partial continuous function and $f = \psi^{-1}\circ g \circ \phi$. Hence version 2 of the Solecki dichotomy above really is a weakened version of the original Solecki dichotomy.

Before going further, let's try to understand this notion of reducibility a little better. Suppose that $f$ is reducible to $g$ via partial continuous functions $\phi$ and $\psi$---i.e.\ that $f = \psi\circ g\circ \phi$. Then the task of evaluating $f$ at a given point can be achieved by evaluating $g$ at a single point, together with some continuous pre- and post-processing using $\phi$ and $\psi$, respectively.

This way of understanding reducibility suggests another, slightly weaker, notion: instead of only allowing $\psi$ to use $g(\phi(x))$ in the post-processing step, why not allow it to use the original input, $x$, as well? Using this weakened notion of reducibility yields our final version of the Solecki dichotomy, which is the version we will use for the rest of the paper.

\begin{definition}
\label{def:weaklyreducible}
A function $f \colon \omega^\omega \to \omega^\omega$ is \term{weakly reducible}\footnote{Also known as \term{continuous Weihrauch reducible}.} to a function $g\colon \omega^\omega \to \omega^\omega$, written $f \leq_w g$, if there are partial continuous functions $\phi \colon \omega^\omega \to \omega^\omega$ and $\psi \colon \omega^\omega \times \omega^\omega \to \omega^\omega$ such that for all $x \in \omega^\omega$, $f(x) = \psi(g(\phi(x)), x)$.
\end{definition}

\begin{theorem}[Solecki dichotomy, version 3]
\label{thm:solecki3}
For every Borel function $f \colon \omega^\omega \to \omega^\omega$, either $f$ is $\sigma$-continuous or $J \leq_w f$.
\end{theorem}

\subsection*{The Posner-Robinson theorem}

Informally, the Posner-Robinson theorem says that every real $x$ is either computable or ``looks like'' $0'$ relative to some real $y$. Formally stated, it reads as follows.

\begin{theorem}[Posner-Robinson theorem]
For every real $x$, either $x$ is computable or there is some real $y$ such that $x \oplus y \geq_T y'$.
\end{theorem}

This theorem was first proved by Posner and Robinson in~\cite{posner1981degrees} and extended by Jockusch and Shore~\cite{jockusch1985rea} and Shore and Slaman~\cite{shore1999defining}. As usual, there is also a relativized version of this theorem.

\begin{theorem}[Posner-Robinson theorem, relativized version]
For every real $z$ and every real $x$, either $x \leq_T z$ or there is some real $y$ such that $x \oplus y \oplus z \geq_T (y \oplus z)'$.
\end{theorem}

We can weaken this relativized version of the Posner-Robinson theorem by requiring the conclusion to hold not for every $z$, but only for all $z$ in some cone of Turing degrees.

\begin{definition}
A \term{cone of Turing degrees} is a set of the form $\{x \in \omega^\omega \mid x \geq_T y\}$ for some fixed real $y$, called the \term{base} of the cone.
\end{definition}

\begin{theorem}[Posner-Robinson theorem, cone version]
\label{thm:pr3}
There is a cone of Turing degrees, $C$, such that for every $z \in C$ and every real $x$, either $x \leq_T z$ or there is some real $y$ such that $x\oplus y\oplus z \geq_T (y\oplus z)'$.
\end{theorem}

\subsection*{Determinacy principles}

As we mentioned earlier, determinacy principles are the main technical tool in our proofs. In this section, we will state the main theorems about determinacy that we need, as well as a useful corollary of these theorems.

Recall that determinacy principles involve games of the following form: two players, called player 1 and player 2, alternate playing natural numbers. Each player can see all previous moves by the other player. At the end, they have jointly formed a sequence $x \in \omega^\omega$. To determine the winner, we have some fixed set $A \subseteq \omega^\omega$, called the \term{payoff set}. Player 1 wins if $x \in A$ and otherwise player 2 wins. The game with payoff set $A$ is sometimes denoted $G(A)$.

In principle, it is possible that for a fixed payoff set $A$, neither player has a winning strategy. When one of the two players does have a winning strategy, the game $G(A)$ is said to be \term{determined}. Determinacy principles assert that when $A$ is sufficiently definable, $G(A)$ must be determined. For example, Martin proved that whenever $A$ is Borel, $G(A)$ is determined \cite{martin1985purely}.

\begin{theorem}[Borel determinacy]
For every Borel set $A \subseteq \omega^\omega$, $G(A)$ is determined.
\end{theorem}

Determinacy principles for sets which are not Borel are typically not provable in $\ZFC$. However, they are usually provable from large cardinal principles. For example, Martin proved that if there is a measurable cardinal then all games with $\mathbf{\Pi}^1_1$ payoff sets are determined \cite{martin2003measurable}.

\begin{theorem}[Analytic determinacy]
Assume that there is a measurable cardinal. Then for every $\mathbf{\Pi}^1_1$ set $A \subseteq \omega^\omega$, $G(A)$ is determined.
\end{theorem}

There is also an axiom, known as the \term{Axiom of Determinacy} and abbreviated $\AD$, which states that for \emph{all} sets $A\subseteq \omega^\omega$, $G(A)$ is determined. This axiom is incompatible with the axiom of choice, but it is consistent with $\ZF$ \cite{koellner2010large}.

\begin{theorem}
Assuming that $\ZF + \text{``there are infinitely many Woodin cardinals''}$ is consistent, so is $\ZF + \AD$.
\end{theorem}

Martin has also proved a corollary of determinacy which is often useful in computability theory and which we will use below. To state it we need a few more definitions.

\begin{definition}
A set $A \subseteq \omega^\omega$ is \term{cofinal in the Turing degrees} (or sometimes just \term{cofinal}) if for every $x \in \omega^\omega$ there is some $y \in A$ such that $y \geq_T x$.
\end{definition}

Note that if a set $A \subseteq \omega^\omega$ does not contain a cone of Turing degrees then its complement must be cofinal (and vice-versa).

\begin{definition}
A \term{pointed perfect tree} is a tree $T$ on $\omega$ such that
\begin{enumerate}
\item $T$ has no dead ends---i.e.\ every node has at least one child.
\item $[T]$ has no isolated paths---i.e.\ every node has incomparable descendants.
\item Every path $x \in [T]$ computes $T$.
\end{enumerate}
\end{definition}

It is not too hard to show that if $T$ is a pointed perfect tree then $[T]$ is cofinal in the Turing degrees. Martin showed that determinacy implies a partial converse of this: if $A$ is cofinal then there is some pointed perfect tree $T$ such that $[T] \subseteq A$~\cite{martin1968axiom, slaman1988definable}. Moreover, the amount of determinacy required to prove this matches the complexity of $A$.

\begin{theorem}
Suppose $A \subseteq \omega^\omega$ is cofinal in the Turing degrees. Then:
\begin{itemize}
\item If $A$ is Borel, $A$ contains $[T]$ for some pointed perfect tree $T$.
\item If $\mathbf{\Pi}^1_1$-determinacy holds and $A$ is $\mathbf{\Pi}^1_1$ then $A$ contains $[T]$ for some pointed perfect tree $T$.
\item If $\AD$ holds and $A$ is any set then $A$ contains $[T]$ for some pointed perfect tree $T$.
\end{itemize}
\end{theorem}

There is a simple observation that yields a useful strengthening of this theorem. Namely, suppose $\{A_n\}_{n \in \omega}$ is a countable sequence of subsets of $\omega^\omega$ whose union is cofinal in the Turing degrees. Then there must be some $n$ such that $A_n$ is cofinal in the Turing degrees. Thus we have the following theorem.

\begin{theorem}
\label{thm:pointedperfecttree}
Suppose $\seq{A_n}_{n \in \omega}$ is a countable sequence such that $\bigcup_{n \in \omega}A_n$ is cofinal in the Turing degrees. Then:
\begin{itemize}
\item If each $A_n$ is Borel then there is some $n$ and pointed perfect tree $T$ such that $A_n$ contains $[T]$.
\item If $\mathbf{\Pi}^1_1$-determinacy holds and each $A_n$ is $\mathbf{\Pi}^1_1$ then there is some $n$ and pointed perfect tree $T$ such that $A_n$ contains $[T]$.
\item If $\AD$ holds then there is some $n$ and pointed perfect tree $T$ such that $A_n$ contains $[T]$.
\end{itemize}
\end{theorem}

\section{Posner-Robinson $\implies$ Solecki}
\label{sec:prtosolecki}

In this section, we will assume a version of the Posner-Robinson theorem (specifically Theorem~\ref{thm:pr3}) and use it to prove a version of the Solecki dichotomy (specifically Theorem~\ref{thm:solecki3}). Here's a brief outline of the proof. First, for any functions $f, g \colon \omega^\omega \to \omega^\omega$, we will introduce a game, $G(f,g)$, and show that player 2 has a winning strategy in this game if and only if $f \leq_w g$. We will then show that in the special case of the game $G(J, f)$, player 1 has a winning strategy if and only if $f$ is $\sigma$-continuous. It is in this step that we will make use of the Posner-Robinson theorem. Finally, it will be clear from the definition of $G(f, g)$ that as long as $f$ and $g$ are both Borel functions then the payoff set of $G(f, g)$ is also Borel. Thus by Borel determinacy, if $f \colon \omega^\omega \to \omega^\omega$ is Borel then either player 1 wins $G(J, f)$ or player 2 does. In the first case, $f$ is $\sigma$-continuous and in the second case, $J \leq_w f$ and so we have reached the dichotomy in the statement of Theorem~\ref{thm:solecki3}.

The game $G(f, g)$ is played as follows: player 2 first plays a code $e \in \omega$ for a three place Turing functional\footnote{Here we use the phrase \term{Turing functional} to indicate a partial computable function on real numbers.}. For the rest of the game, player 1 plays a real $x$ and player 2 plays two reals, $y$ and $z$. On every turn, player 1 will play one more digit of the real $x$ and player 2 will play one more digit of the the real $z$ and either one or zero more digits of the real $y$. Player 2 wins if they eventually play all digits of $y$ (in other words, player 2 can delay arbitrarily long between playing one digit of $y$ and the next but if they delay forever then they forfeit the game) and $f(x) = \Phi_e(g(y), x, z)$. Otherwise, player 1 wins. The game can be pictured as follows.
\vspace{5pt}
\begin{align*}
\begin{array}{c | c c c c c c c c c c l}
  \text{player } 1 &   & x_0 &          & x_1 &     & \ldots & x_n &          & \ldots &\quad\quad& x = x_0x_1x_2\ldots\\
  \cline{1-10}
  \text{player } 2 & e &     & y_0, z_0 &     & z_1 & \ldots &     & y_1, z_n & \ldots &\quad\qquad& y = y_0y_1y_2\ldots,\quad z = z_0z_1z_2\ldots
\end{array}\\
\end{align*}

We can understand this game as follows. First recall that $f \leq_w g$ means that there are partial continuous functions $\phi$ and $\psi$ such that for all $x$, $f(x) = \psi(g(\phi(x)), x)$. Further, recall that a continuous function is just a computable function relative to some oracle. Hence for some code $e \in \omega$ for a Turing functional and some $z \in \omega^\omega$, $\psi(x, y) = \Phi_e(x, y, z)$.

In the game $G(f, g)$, player 2 is trying to convince player 1 that $f \leq_w g$. The natural number $e$ and real $z$ played by player 2 should be thought of as specifying the continuous function $\psi$. Player 1's moves consist of a challenge input $x$ for $f$ and the real $y$ played by player 2 corresponds to $\phi(x)$. The winning condition for player 2---that player 2 plays infinitely many digits of $y$ and that $f(x) = \Phi_e(g(y), x, z)$---corresponds to the reduction procedure implicitly specified by player 2 working successfully on input $x$.

\begin{lemma}
\label{lem:abovejump}
Player 2 wins $G(f, g)$ if and only if $f \leq_w g$.
\end{lemma}

\begin{proof}
$(\implies)$ Suppose that player 2 wins $G(f, g)$ via the strategy $\tau$. Then the following two procedures describe partial continuous functions $\omega^\omega \to \omega^\omega$ and $\omega^\omega\times\omega^\omega \to \omega^\omega$, respectively.
\begin{enumerate}
\item Given $x$ as input: play the game $G(f, g)$, using the digits of $x$ as player 1's moves and using $\tau$ to generate player 2's moves. Output the first of the two reals played by player 2 (the real referred to as $y$ in the description of the game above).
\item Given $w$ and $x$ as input: play the game $G(f, g)$ using the digits of $x$ as player 1's moves and using $\tau$ to generate player 2's moves. Let $e$ be number played by $\tau$ on the first move of the game and let $z$ be the second of the two reals played by player 2. Output $\Phi_e(w, x, z)$.
\end{enumerate}
Let $\phi$ and $\psi$, respectively, denote these two continuous functions. Then the fact that $\tau$ is a winning strategy for player 2 in the game $G(f, g)$ ensures that for all $x$, $f(x) = \psi(g(\phi(x)), x)$.

$(\impliedby)$ Suppose that $f \leq_w g$ via partial continuous functions $\phi$ and $\psi$. Let $e \in \omega$ and $z\in\omega^\omega$ be such that $\psi$ is computed by the $e^{\text{th}}$ Turing functional with oracle $z$. Then the following is a winning strategy for player 2 in the game $G(f, g)$. On the first turn, play the number $e$. On each subsequent turn, play one more digit of $z$. Also on each of these turns, if player 1 has played enough digits of $x$ to determine one more digit of $\phi(x)$, play that as well.
\end{proof}

We now turn to the case where player 1 wins and the game is of the form $G(J, f)$ for some $f$. We will show that in this case, $f$ must be $\sigma$-continuous. To prove this, we first need the following observation about $\sigma$-continuity.

\begin{observation}
\label{obs:sigmacontinuous}
A function $f \colon \omega^\omega\to\omega^\omega$ is $\sigma$-continuous if and only if there is some real $z$ such that for all $x$, $f(x) \leq_T x\oplus z$.
\end{observation}

\begin{proof}
If $f$ is $\sigma$-continuous then it can be written as a countable union of partial continuous functions. Each partial continuous function on $\omega^\omega$ is a partial computable function relative to some oracle, so we can find a countable sequence of codes for Turing functionals $\{e_n\}_{n \in \omega}$ and oracles $\{z_n\}_{n \in \omega}$ such that for each $x$, $f(x) = \Phi_{e_n}(x, z_n)$ for some $n$. So if we take $z = \bigoplus_{n \in \omega}z_n$ then for each $x$, $f(x) \leq_T x\oplus z$.

Conversely, suppose there is some $z$ such that for all $x$, $f(x) \leq_T x \oplus z$. For each $n$, define $A_n = \{x \mid f(x) = \Phi_n(x, z)\}$. Then $\{A_n\}_{n \in \omega}$ is a countable sequence of sets whose union covers $\omega^\omega$. Also, for each $n$, $f\uh_{A_n}$ is computable relative to $z$ via $\Phi_n$ and hence continuous. So $f$ is $\sigma$-continuous.
\end{proof}

\begin{lemma}
\label{lem:sigmacontinuous}
Player 1 wins $G(J, f)$ if and only if $f$ is $\sigma$-continuous.
\end{lemma}

\begin{proof}
$(\implies)$ Suppose that player 1 wins $G(J, f)$ via the strategy $\sigma$. Let $w$ be the base of a cone for which the conclusion of Theorem~\ref{thm:pr3} applies (i.e.\ such that the Posner-Robinson theorem holds relative to every real which computes $w$).

We claim that for every $y$, $f(y) \leq_T y \oplus \sigma \oplus w$ and hence $f$ is $\sigma$-continuous by Observation~\ref{obs:sigmacontinuous}. To prove this, we will show that if not then $\sigma$ is not actually a winning strategy for player 1. So suppose that there is some $y$ such that $f(y) \nleq_T y \oplus \sigma \oplus w$. Since the Posner-Robinson theorem holds relative to $y\oplus \sigma\oplus w$, we can find some real $v$ such that $f(y)\oplus v\oplus y\oplus \sigma\oplus w \geq_T (v\oplus y\oplus \sigma\oplus w)'$.

We will now explain how to win while playing as player 2 in $G(J, f)$ against the strategy $\sigma$. We will play as follows: first we play some number $e \in \omega$, which we will explain how to choose later. Then we ignore player 1's moves entirely and play the reals $y$ and $z = v \oplus y\oplus \sigma \oplus w$. Note that from $z$ we can compute player 1's moves since $z$ computes both player 1's strategy $\sigma$ and all of player 2's moves. In other words, if $x$ is the real played by player 1 then $x \leq_T z$. Hence by our choice of $z$ we have
\[
  x' \leq_T z' \leq_T f(y)\oplus z.
\]
This is almost what we want, but there is one problem: for player 2 to win, we need not just that $f(y) \oplus z$ computes $x'$, but that it does so via the Turing functional specified by player 2. And the only problem with this is that the Turing functional which computes $x'$ from $f(y)$ and $z$ depends on the code $e$ played by player 2. However, we can get around this by using the recursion theorem.

In precisely, note that while the computation of $x'$ from $f(y)\oplus z$ depends on the value of $e$ (because the value of $x$ itself depends on $e$), the dependence is uniform in $e$\footnote{This is because the computation of $z'$ from $f(y) \oplus z$ does not depend on $e$ and the computation of $x$ from $z$, and hence $x'$ from $z'$ is uniform in $e$.}. In other words, there is some $a \in \omega$ such that for all $e$,
\[
  \Phi_a(f(y), z, e) = (\sigma * (e, y, z))'
\]
where by $(\sigma * (e, y, z))$ we mean the real played by the strategy $\sigma$ in response to player 2 playing $e$ along with the reals $y$ and $z$. Thus by the recursion theorem, we can find some $e$ such that
\[
  \Phi_e(f(y), z) = \Phi_a(f(y), z, e) = (\sigma * (e, y, z))'.
\]
Thus by playing this $e$ as our first move as player 2 (and then playing $y$ and $z$), we can win against $\sigma$.

$(\impliedby)$ Suppose that $f$ is $\sigma$-continuous. Thus by our observation, there is some $w$ such that for all $x$, $f(x) \leq_T x \oplus w$. Consider the following strategy for player 1 in the game $G(J, f)$: alternate playing digits of $w$ and copying the moves played by player 2. We can picture this strategy as follows.
\vspace{5pt}
\begin{align*}
\begin{array}{c | c c c c c c c c c}
  \text{player } 1 &   & w_0 &          & \ip{y_0}{z_0} &     & w_1 &          & \seq{z_1} & \ldots \\
  \cline{1-10}
  \text{player } 2 & e &     & y_0, z_0 &               & z_1 &     & y_1, z_2 &     & \ldots 
\end{array}\\
\end{align*}

We claim this is a winning strategy for player 1. To see why, suppose player 1 follows this strategy and that player 2 plays $e \in \omega$ and $y, z \in \omega^\omega$. Then the real $x$ played by player 1 will compute $w\oplus y\oplus z$. Player 2 can only win if $\Phi_e(f(y), x, z) = J(x)$, but this is impossible since we have
\[
  f(y) \leq_T y \oplus w \leq_T x
\]
and hence $\Phi_e(f(y), x, z) \leq_T x$, but $J(x) \nleq_T x$.
\end{proof}

We can now finish our proof of the Solecki dichotomy from the Posner-Robinson theorem.

\begin{proof}[Proof of Theorem~\ref{thm:solecki3} from Theorem~\ref{thm:pr3}]
Let $f\colon \omega^\omega \to \omega^\omega$ be a Borel function and consider the game $G(J, f)$. By Borel determinacy, either player 1 or player 2 has a winning strategy for this game. In the former case, $f$ is $\sigma$-continuous by Lemma~\ref{lem:sigmacontinuous}. In the latter case, $J \leq_w f$ by Lemma~\ref{lem:abovejump}.
\end{proof}

The results of this section raise an obvious question. Namely, is it possible to use the Posner-Robinson theorem to prove the full Solecki dichotomy (i.e.\ either Theorem~\ref{thm:solecki1} or Theorem~\ref{thm:solecki2})? Let us mention one possible route to such a proof. In this section, we described a game, $G(f, g)$, which can be used to characterize weak reducibility of $f$ to $g$ and this game was the key to our proof of Theorem~\ref{thm:solecki3}. It seems plausible that finding a game which characterizes reducibility of $f$ to $g$, rather than weak reducibility, would yield a proof of Theorem~\ref{thm:solecki2}. Finding such a game might also be of independent interest.

\begin{question}
Can the Posner-Robinson theorem together with Borel determinacy be used to prove either Theorem~\ref{thm:solecki1} or Theorem~\ref{thm:solecki2}?
\end{question}

\begin{question}
Is there a game characterizing reducibility in the same way that the game $G(f, g)$ described above characterizes weak reducibility?
\end{question}

\section{Solecki $\implies$ Posner-Robinson}
\label{sec:soleckitopr}

In this section we will assume a version of the Solecki dichotomy (specifically Theorem~\ref{thm:solecki3}) and use it to prove a version of the Posner-Robinson theorem (specifically Theorem~\ref{thm:pr3}). We will do so by proving the contrapositive: we will show that if Theorem~\ref{thm:pr3} fails then so does Theorem~\ref{thm:solecki3}. Also, as we mentioned in the introduction, our proof is carried out assuming $\mathbf{\Pi}^1_1$-determinacy, a statement which is not provable in $\ZFC$, but which is provable from $\ZFC$ plus the existence of a measurable cardinal.

The core idea of the proof is very simple---it essentially consists of the observation that if $f$ is a function which takes each real $x$ to a witness of the failure of the Posner-Robinson theorem relative to $x$ then $f$ does not satisfy the conclusion of the Solecki dichotomy. However, this simple idea is complicated by the need to make sure $f$ is Borel (otherwise we cannot invoke Theorem~\ref{thm:solecki3}). Most of the details of the proof will be devoted to overcoming this obstacle.

We will now go into the details of the proof. Suppose that Theorem~\ref{thm:pr3} fails. In other words, the set
\[
  A = \{x \in \omega^\omega \mid \text{the Posner-Robinson theorem holds relative to } x\}
\]
does not contain any cone of Turing degrees. Hence its complement, the set
\[
  B = \{x \in \omega^\omega \mid \text{the Posner-Robinson theorem fails relative to } x\},
\]
is cofinal in the Turing degrees.

Now suppose that we can find a function $f \colon B \to \omega^\omega$ such that for each $x$ in $B$, $f(x)$ is a witness to the failure of the Posner-Robinson theorem relative to $x$---i.e.\ $f(x) \nleq_T x$ and there is no $y$ such that $f(x) \oplus y \oplus x \geq_T (y \oplus x)'$. Extend $f$ to a total function on $\omega^\omega$ by setting $f(x) = 0$ for all $x \notin B$. Note that this modified version of $f$ has the following two properties. 
\begin{enumerate}[(1)]
\item For cofinally many $x$, $f(x) \nleq_T x$.
\item For all $x$, there is no $y$ such that $f(x) \oplus y \oplus x \geq_T (y \oplus x)'$. When $x \in B$ this is by assumption and when $x \notin B$ this is because $f(x)$ is computable.
\end{enumerate}
The next two lemmas show that these properties imply that $f$ is a counterexample to the Solecki dichotomy.

\begin{lemma}
\label{lem:counterexample1}
Suppose that $f \colon \omega^\omega \to \omega^\omega$ is a function such that for a set of $x$ which is cofinal in the Turing degrees, $f(x) \nleq_T x$. Then $f$ is not $\sigma$-continuous. 
\end{lemma}

\begin{proof}
For contradiction, assume $f$ is $\sigma$-continuous. By Observation~\ref{obs:sigmacontinuous}, there must be some $z$ such that for every $x$, $f(x) \leq_T x \oplus z$. By assumption, we can find some $x \geq_T z$ such that $f(x) \nleq_T x$. But since $x \equiv_T x \oplus z$, this implies $f(x) \nleq_T x\oplus z$, which contradicts our choice of $z$.
\end{proof}

\begin{lemma}
\label{lem:counterexample2}
Suppose that $f \colon \omega^\omega \to \omega^\omega$ is a function such that for all $x$, there is no $y$ such that $f(x) \oplus y \oplus x \geq_T (y \oplus x)'$. Then $J \nleq_w f$.
\end{lemma}

\begin{proof}
  For contradiction, assume that $J \leq_w f$. Thus there are partial continuous functions $\phi$ and $\psi$ such that for all $x$, $J(x) = \psi(f(\phi(x)), x)$. Let $z$ be an oracle relative to which $\psi$ and $\phi$ are computable. Let $x = \phi(z)$. By assumption, there is no $y$ such that $f(x)\oplus y \oplus x \geq_T (y \oplus x)'$. We claim this is contradicted by taking $y = z$.

To see why, note that since $\phi$ and $\psi$ are computable relative to $z$, we have
\[
  x = \phi(z) \leq_T z \qquad\text{and}\qquad \psi(f(x), z) \leq_T f(x)\oplus z
\]
and by our choice of $\phi$ and $\psi$ we have
\[
  z' = \psi(f(\phi(z)), z) = \psi(f(x), z).
\]
Hence we have
\[
  (z\oplus x)' \equiv_T z' = \psi(f(x), z) \leq_T f(x)\oplus z \leq_T f(x)\oplus z\oplus x.
\]
which yields the contradiction.
\end{proof}

We are now left with the problem of finding some $f\colon B \to \omega^\omega$ with the properties described above. Of course, it is easy to find such an $f$ using the Axiom of Choice, but there is no reason to believe that a function chosen in this way will be Borel and hence we cannot apply Theorem~\ref{thm:solecki3}. Instead of using choice, we could try to appeal to a uniformization theorem from descriptive set theory. However, the relation that we need to uniformize, namely
\[
  \{(x, y) \mid y \nleq_T x \text{ and } \forall z\, (y\oplus z \oplus x \ngeq_T (z \oplus x)'\},
\]
is $\Pi^1_1$ and thus too complicated for any of the standard uniformization theorems to give us a Borel---or even analytic---uniformizing function.

We will now see how to find some $f$ with the necessary properties which is Borel (in fact, Baire class 1). The key step is the following lemma.

\begin{lemma}
\label{lem:lowfailure}
Suppose that for cofinally many $x$, the Posner-Robinson theorem fails relative to $x$. Then for cofinally many $x$, there is some $y \leq_T x'$ which witnesses this failure---i.e.\ such that $y \nleq_T x$ and there is no $z$ such that $y\oplus z\oplus x \geq_T (z\oplus x)'$.
\end{lemma}

\begin{proof}
Let $w \in \omega^\omega$ be arbitrary. We need to show that there is some $x \geq_T w$ such that the Posner-Robinson theorem fails relative to $x$ and such that this failure is witnessed by some $y \leq_T x'$. By increasing $w$ if necessary (and invoking our assumption that the Posner-Robinson theorem fails cofinally), we may assume the Posner-Robinson theorem fails relative to $w$. Let $y$ be a witness to the failure of the Posner-Robinson theorem relative to $w$.

The key observation is that it is sufficient to find some $x \geq_T w$ such that $x'$ computes $y$ but $x$ does not. Suppose we can find such an $x$. We claim that the Posner-Robinson theorem fails relative to $x$ and that this failure is witnessed by $y$. If not, then we can find some $z$ such that $y\oplus z\oplus x\geq_T (z\oplus x)'$. But since $x \geq_T w$, this gives us
\[
  y\oplus z\oplus x \oplus w \equiv_T y\oplus z\oplus x \geq_T (z\oplus x)' \equiv_T (z\oplus x\oplus w)'
\]
and hence $y$ is \emph{not} a witness to the failure of the Posner-Robinson theorem relative to $w$.

We will now explain how to construct $x$. In fact, we will actually construct a real $x_0$ such that $(x_0 \oplus w)'$ computes $y$ but $x_0 \oplus w$ does not and then set $x = x_0\oplus w$. The construction is similar to the proof of Friedberg jump inversion. We will construct $x_0$ by finite initial segments. On step $e$ of the construction, we will make sure that $\Phi_e^{x_0\oplus w}$ does not correctly compute $y$ and then code one more digit of $y$.

Suppose we are on step $e$ of the construction and the initial segment of $x_0$ that we have built so far is $\sigma_e$. There are two cases to consider.
\begin{enumerate}[(1)]
\item There is some $n \in \omega$ and strings $\tau_0, \tau_1$ such that $\Phi_e^{\sigma_e\concat\tau_0\oplus w}(n)\downarrow \neq \Phi_e^{\sigma_e\concat\tau_1\oplus w}(n)\downarrow$. In this case, one of these two values must disagree with $y(n)$. Let $\langle n, \tau_0, \tau_1\rangle$ be the first such triple discovered in some $w$-computable search and set $\sigma_{e + 1} = \sigma_e\concat\tau_i\concat\seq{y(e)}$ where $\tau_i$ is equal to whichever of $\tau_0, \tau_1$ causes $\Phi_e^{x_0\oplus w}(n)$ to disagree with $y(n)$.
\item For every $n \in \omega$, there is at most one value of $\Phi_e^{\sigma_e\concat\tau\oplus w}(n)$ obtainable over all strings $\tau$. Then by standard arguments, if $x_0$ extends $\sigma_e$ then either $\Phi_e^{x_0\oplus w}$ is not total or it computes a real which is computable from $w$ alone. In either case, $\Phi_e^{x_0\oplus w}$ cannot be equal to $y$ so we may simply set $\sigma_{e + 1} = \sigma_e\concat\seq{y(e)}$.
\end{enumerate}

It is clear from the construction that $x_0\oplus w$ does not compute $y$. To see that $(x_0 \oplus w)'$ computes $y$, simply note that $(x_0\oplus w)'$ can figure out what happened at each step of the construction described above (i.e.\ it can check which of the two cases held at each step and, in the first case, recover the triple $\langle n, \tau_0, \tau_1\rangle$) and can thus recover the digits of $y$ coded during the construction.
\end{proof}

We can now prove Theorem~\ref{thm:pr3}. As we mentioned above, our proof uses $\mathbf{\Pi}^1_1$-determinacy.

\begin{proof}[Proof of Theorem~\ref{thm:pr3} from Theorem~\ref{thm:solecki3}]
Suppose Theorem~\ref{thm:pr3} fails. Then as we saw above, the set
\[
  B = \{x \in \omega^\omega \mid \text{the Posner-Robinson theorem fails relative to $x$}\}
\]
is cofinal in the Turing degrees and hence by Lemma~\ref{lem:lowfailure}, the following set is also cofinal
\begin{align*}
  C = \{x \in \omega^\omega \mid \exists y \leq_T x'\, (&y \text{ witnesses the failure of the}\\
  &\text{Posner-Robinson theorem relative to $x$})\}.
\end{align*}
Now for each $e \in \omega$, define
\begin{align*}
  C_e = \{x \in \omega^\omega \mid \Phi_e(x') &\text{ is total and witnesses the failure of the}\\
  &\text{ Posner-Robinson theorem relative to $x$}\}
\end{align*}
and note that $C = \bigcup_{e \in \omega}C_e$. Hence by Theorem~\ref{thm:pointedperfecttree}, there is some pointed perfect tree $T$ and $e \in \omega$ such that $[T] \subseteq C_e$.

Define $f \colon \omega^\omega \to \omega^\omega$ by
\[
  f(x) =
  \begin{cases}
    \Phi_e(x') &\text{if } x \in [T]\\
    0 &\text{else.}
  \end{cases}
\]
Let's now make some observations about $f$.
\begin{enumerate}[(1)]
\item $f$ is clearly Borel---in fact, it is actually Baire class $1$.
\item For every $x \in [T]$, $f(x)$ is a witness to the failure of the Posner-Robinson theorem relative to $x$.
\item In particular, for any $x \in [T]$, $f(x) \nleq_T x$. Since $[T]$ is a cofinal set in the Turing degrees, $f$ satisfies the hypothesis of Lemma~\ref{lem:counterexample1}.
\item For any $x \in \omega^\omega$, there is no $y$ such that $f(x)\oplus y \oplus x \geq_T (y\oplus x)'$. If $x \in [T]$ then this is because $f(x)$ is a witness to the failure of the Posner-Robinson theorem relative to $x$. If $x \notin [T]$ then this is because $f(x)$ is computable. Hence $f$ satisfies the hypothesis of Lemma~\ref{lem:counterexample2}.
\end{enumerate}
Thus by Lemmas~\ref{lem:counterexample1} and~\ref{lem:counterexample2}, $f$ is a counterexample to Theorem~\ref{thm:solecki3}.
\end{proof}

\section{Generalizations}
\label{sec:generalizations}

Recently, Marks and Montalb\'{a}n have generalized the Solecki dichotomy (specifically, Theorem~\ref{thm:solecki2}) to higher levels of the Borel hierarchy \cite{marks2021dichotomy}. To state their result, we must introduce a few more definitions. First, we must generalize $\sigma$-continuity. Recall that for any countable ordinal $\alpha$, a function $f$ is \term{$\mathbf{\Sigma}^0_\alpha$-measurable} if for every open set $U$, $f^{-1}(U)$ is in $\mathbf{\Sigma}^0_\alpha$.

\begin{definition}
For any countable ordinal $\alpha$,
\begin{itemize}
\item $\Dec_\alpha$ denotes the set of functions $f \colon \omega^\omega \to \omega^\omega$ for which there is a partition $\{A_n\}_{n \in \omega}$ of $\omega^\omega$ into countably many pieces such that for each $n$, $f\uh_{A_n}$ is $\mathbf{\Sigma}^0_\alpha$-measurable with respect to the subspace topology on $A_n$.
\item $\Dec_{< \alpha}$ denotes the set of functions $f \colon \omega^\omega \to \omega^\omega$ for which there is a partition $\{A_n\}_{n \in \omega}$ of $\omega^\omega$ into countably many pieces such that for each $n$, $f \uh_{A_n}$ is $\mathbf{\Sigma}^0_\beta$-measurable for some $\beta < \alpha$ (note that $\beta$ may depend on $n$).
\end{itemize}
\end{definition}

Note that $\mathbf{\Sigma^0_1}$-measurable is the same as continuous, so a function $f$ is $\sigma$-continuous if and only if it is $\Dec_1 = \Dec_{< 2}$.

As a warning to readers, the notation in this area is not yet standardized and our notation does not quite match previously used notation. In particular, $\Dec_\alpha$ is sometimes denoted $\Dec(\mathbf{\Sigma}^0_\alpha)$ (see e.g.~\cite{gregoriades2021turing}). We have chosen the notation $\Dec_\alpha$ to be consistent with our chosen notation for $\Dec_{< \alpha}$, for which there does not seem to currently be any standard notation but which is necessary to correctly express Marks and Montalb\'{a}n's generalization of the Solecki dichotomy at limit levels of the Borel hierarchy.

For each countable ordinal $\alpha \geq 1$, we will use $J_\alpha \colon \omega^\omega \to \omega^\omega$ to denote the $\alpha^\text{th}$ Turing jump---i.e.\ $J_\alpha(x) = x^{(\alpha)}$. Note that technically $J_\alpha$ depends on a choice of presentation for $\alpha$. However, the versions of $J_\alpha$ that are obtained by choosing different presentations for $\alpha$ are all reducible to each other (in the sense of Definition~\ref{def:reducible}) and so this subtlety does not matter for us.

We can now state the appropriate generalization of Theorem~\ref{thm:solecki2} due to Marks and Montalb\'{a}n.

\begin{theorem}[Generalized Solecki dichotomy]
\label{thm:soleckigeneral}
For every countable ordinal $\alpha \geq 1$ and every Borel function $f\colon \omega^\omega \to \omega^\omega$, either $f$ is in $\Dec_{<(1 + \alpha)}$ or $J_\alpha \leq f$.
\end{theorem}

There is also a generalization of the Posner-Robinson theorem to higher levels of the hyperarithmetic hierarchy, due to Slaman and Shore \cite{shore1999defining}.

\begin{theorem}[Generalized Posner-Robinson theorem]
For all computable ordinals $\alpha$ and all reals $x$, either $x \leq_T 0^{(\beta)}$ for some $\beta < \alpha$ or there is some real $y$ such that $x \oplus y \geq_T y^{(\alpha)}$.
\end{theorem}

As usual, there is also a relativized version.

\begin{theorem}
\label{thm:prgeneral}
For all reals $z$, all ordinals $\alpha$ which are computable relative to $z$ and all reals $x$, either $x \leq_T z^{(\beta)}$ for some $\beta < \alpha$ or there is some real $y$ such that $x \oplus y \oplus z \geq_T (y \oplus z)^{(\alpha)}$.
\end{theorem}

The main results of this paper also hold for these generalizations of the Solecki dichotomy and the Posner-Robinson theorem. In particular, we can introduce weakened versions of Theorems~\ref{thm:soleckigeneral} and~\ref{thm:prgeneral}:

\begin{theorem}
For every countable ordinal $\alpha \geq 1$ and every Borel function $f\colon \omega^\omega \to \omega^\omega$, either $f$ is in $\Dec_{<(1 + \alpha)}$ or $J_\alpha \leq_w f$.
\end{theorem}

\begin{theorem}
For every countable ordinal $\alpha$, there is some cone of Turing degrees, $C$, such that for all $z \in C$ and all reals $x$, $\alpha$ is computable relative to $z$ and either $x \leq_T z^{(\beta)}$ for some $\beta < \alpha$ or there is some real $y$ such that $x \oplus y \oplus z \geq_T (y \oplus z)^{(\alpha)}$.
\end{theorem}

The proofs in sections~\ref{sec:prtosolecki} and~\ref{sec:soleckitopr} work, with almost no modifications, to show that the two theorems above are equivalent.

\section*{Acknowledgments}

Thanks to Andrew Marks and Ted Slaman for useful conversations and advice.

\bibliographystyle{plain}
\bibliography{main}

\end{document}